\documentclass[10pt]{article}
\usepackage[utf8]{inputenc}
\usepackage[english]{babel}
\usepackage{amsmath}
\usepackage{amssymb}
\usepackage{amsthm}
\usepackage{algorithm,algorithmicx,algpseudocode}
\usepackage{verbatim}
\usepackage{graphicx,color}
\usepackage{graphics}
\usepackage{subfigure}
\usepackage{mathrsfs}
\usepackage{epstopdf}

\numberwithin{equation}{section}
\theoremstyle{plain}
\newtheorem{theorem}{Theorem}
\newtheorem{lemma}[theorem]{Lemma}

\newtheorem{corollary}[theorem]{Corollary} 
\theoremstyle{definition}
\newtheorem{definition}[theorem]{Definition}

\theoremstyle{remark}
\newtheorem{rem}[theorem]{Remark}

\newcommand{\T}{\mathcal{T}}
\newcommand{\E}{\mathcal{E}}

\newcommand{\V}{\mathcal{V}}

\newcommand{\fgref}[1]{Figure~\ref{#1}}

\newcommand{\lemref}[1]{Lemma~\ref{#1}}
\newcommand{\corref}[1]{Corollary~\ref{#1}}
\newcommand{\thmref}[1]{Theorem~\ref{#1}}
\newcommand{\Enormh}[1]{||| #1 |||_h}
\newcommand{\Enormhw}[2]{||| #1 |||_{h,#2}}
\newcommand{\EnormH}[1]{||| #1 |||_H}
\newcommand{\EnormHw}[2]{||| #1 |||_{H,#2}}

\newcommand{\e}{\mathscr{E}}
\newcommand{\Bh}[2]{a_h(#1,#2)}
\newcommand{\mscale}{A}
\newcommand{\amin}{\alpha}
\newcommand{\amax}{\beta}
\newcommand{\VH}{\mathcal{V}_H}

\newcommand{\VmshH}{\mathcal{V}^{ms}_{H}}
\newcommand{\VmshHL}{\mathcal{V}^{ms,L}_{H}}
\newcommand{\WmshHL}{\mathcal{W}^{ms,L}_{H}}
\newcommand{\Vh}{\mathcal{V}_h}
\newcommand{\Vfh}{\mathcal{V}^{\operatorname*{f}}} 
\newcommand{\Vfhw}[1]{\mathcal{V}^f_h(#1)}

\newcommand{\Ich}{\mathcal{I}^{\operatorname*{c}}_h}
\newcommand{\IcH}{\mathcal{I}^{\operatorname*{c}}_H}
\newcommand{\Lp}{\Pi_H}

\newcommand{\dx}{\operatorname*{d}\hspace{-0.3ex}x}

\newcommand{\F}{\mathfrak{F}}
\newcommand{\umshH}{u^{ms}_{H}}

\newcommand{\umshHT}{u^{ms}_{H,T}}
\newcommand{\umshHL}{u^{ms,L}_{H}}

\newcommand{\wmshHL}{w^{ms,L}_{H}}

\newcommand{\ba}{C_{\amax/\amin}}

\newcommand{\mc}{\mathcal}
\newcommand{\diam}{\operatorname*{diam}}
\newcommand{\dime}[1]{\operatorname*{dim}(#1)}
\newcommand{\supp}{\text{supp}}

\begin{document} 
\title{Convergence of a \\ discontinuous Galerkin multiscale method}

\author{Daniel Elfverson\footnotemark[2]\ \footnotemark[5] \and Emmanuil H. Georgoulis\footnotemark[3] \and Axel Målqvist\footnotemark[2]\ \footnotemark[5] \and Daniel Peterseim\footnotemark[4]\ \footnotemark[6]}

\renewcommand{\thefootnote}{\fnsymbol{footnote}}
\footnotetext[2]{Information Technology, Uppsala University, Box 337, SE-751 05, Uppsala, Sweden.}
\footnotetext[3]{Department of Mathematics, University of Leicester, Leicester, UK.}
\footnotetext[4]{Institut für Mathematik, Humboldt-Universität zu Berlin, Berlin, Germany.}
\footnotetext[5]{Supported by The Swedish Research Council and The Göran Gustafsson Foundation.}
\footnotetext[6]{Supported by the DFG Research Center Matheon Berlin through project C33.}
\renewcommand{\thefootnote}{\arabic{footnote}}
\pagestyle{myheadings}
\thispagestyle{plain}





\maketitle
\begin{abstract} A convergence result for a discontinuous Galerkin multiscale method for a second order elliptic problem  is presented. We consider a heterogeneous and highly varying diffusion coefficient in $L^\infty(\Omega,\mathbb{R}^{d\times d}_{sym})$ with uniform spectral bounds and without any assumption on scale separation or periodicity. The multiscale method uses a corrected basis that is computed on patches/subdomains. The error, due to truncation of corrected basis, decreases exponentially with the size of the patches. Hence, to achieve an algebraic convergence rate of the multiscale solution on a uniform mesh with mesh size $H$ to a reference solution, it is sufficient to choose the patch sizes as $\mathcal{O}(H|\log(H^{-1})|)$. We also discuss a way to further localize the corrected basis to element-wise support leading to a slight increase of the dimension of the space. Improved convergence rate can be achieved depending on the piecewise regularity of the forcing function. Linear convergence in energy norm and quadratic convergence in $L^2$-norm is obtained independently of the forcing function. A series of numerical experiments confirms the theoretical rates of convergence. 
\end{abstract}

\def\ams{\vspace{.5em}
{\textbf{AMS subject classiﬁcations}:\,\relax%
}}
\def\endams{\par}

\def\keywords{\vspace{.5em}
{\textbf{Keywords}:\,\relax%
}}
\def\endkeywords{\par}

\def\AMS{\vspace{.5em}
{\textbf{AMS subject classifications}:\,\relax%
}}
\def\endAMS{\par}

\begin{keywords}
  multiscale method, discontinuous Galerkin, a priori error estimate, convergence  
\end{keywords}
\begin{AMS}
   65N12, 65N30
\end{AMS}

 \section{Introduction}
This work considers the numerical solution of second order elliptic problems with heterogeneous and highly varying (non-periodic) diffusion coefficient. The heterogeneities and oscillations of the coefficient may appear on several non-separated scales. 
More specifically, let $\Omega\subset\mathbb{R}^{d}$ be a bounded Lipschitz domain with polygonal boundary $\Gamma$. The boundary $\Gamma$ may be partitioned into some non-empty closed subset $\Gamma_D$ (the Dirichlet boundary) and its complement $\Gamma_N:=\Gamma\setminus\Gamma_D$ (the, possibly empty, Neumann boundary).
We assume that the diffusion matrix $A\in L^\infty\left(\Omega,\mathbb{R}_{\mathrm{sym}}^{d\times d}\right)$ has uniform spectral bounds $0<\alpha,\beta<\infty$, defined by
\begin{equation}\label{eq:ba}
0<\alpha:=\underset{x\in\Omega}{\operatorname{ess}\inf}\hspace{-1ex}
\inf\limits_{v\in\mathbb{R}^{d}\setminus\{0\}}\hspace{-2ex}\dfrac{(A( x)v)\cdot v}{v\cdot v}\leq\underset{x\in\Omega}{\operatorname{ess}\sup}\hspace{-1ex}
\sup\limits_{v\in\mathbb{R}^{d}\setminus\{0\}}\hspace{-2ex}\dfrac{(A( x)v)\cdot v
}{v\cdot v}=:\beta<\infty.
\end{equation}
Given $f\in L^{2}( \Omega) $, we seek the weak solution of the boundary-value problem
\begin{equation}\notag
  \begin{aligned}
    -\nabla \cdot\mscale \nabla u &= f\quad \text{in }\Omega, \\
    u & = 0 \quad \text{on }\Gamma_D, \\
    \nu\cdot\mscale\nabla u &= 0\quad \text{on }\Gamma_N,
  \end{aligned}
\end{equation}
i.e., we seek $u\in H^1_D(\Omega):=\{v\in H^1(\Omega)\;\vert\;v\vert_{\Gamma_D}=0\text{ in the sense of traces}\}$, such that
\begin{equation}\label{e:modelvar}
a\left(  u,v\right):=\int_{\Omega}A\nabla u\cdot \nabla
v\dx =\int_{\Omega}fv\dx=:F( v) \quad\text{for all } v\in H^1_D(\Omega).
\end{equation}

Many methods have been developed in recent years to overcome the lack of performance of classical finite element methods in cases where $A$ is rough, meaning that $A$ has discontinuities and/or high variation; we refer to \cite{MR701094,MR1286212,MR1660141,MR1455261,resbub} amongst others. Common to all the aforementioned approaches is the idea to solve the problems on small subdomains and to use the results to construct a better basis for some Galerkin method or to modify the coarse scale operator. However, apart from the one-dimensional setting, the performance of those methods correlates strongly with periodicity and scale separation of the diffusion coefficient. 

Other approaches \cite{MR2721592,OwhadiZ11,BabLip11} perform well without any assumptions of periodicity or scale separation in the diffusion coefficient at the price of a high computational cost: in \cite{MR2721592,OwhadiZ11} the support of the modified basis functions is large and in \cite{BabLip11} the computation of the basis functions involves the solutions of local eigenvalue problems. 

Only recently in \cite{MP11}, a variational multiscale method has been developed that allows for textbook convergence with respect to the mesh size $H$,
$\left\Vert u-u_{H}\right\Vert _{H^{1}(\Omega) }\leq C_{f,\amax/\amin}H$ with a constant $C_{f,\amax/\amin}$ that depends on $f$ and the global bounds of the diffusion coefficient but not its variations. This result is achieved by an operator-dependent modification of the classical nodal basis based on the solution of local problems on vertex patches of diameter $\mc{O}(H|\log(H^{-1})|)$. The method  in \cite{MP11} is an extension of the method presented in \cite{Larson20072313,malqvist2010}.

In this work, we present a discontinuous Galerkin (dG) multiscale method with similar performance. The method is a slight variation of the method \cite{EGM12}, in the sense that the boundary conditions for the local problems are now of Dirichlet type. The dG finite element method admits good conservation properties of the state variable, and also offers the use of very general meshes due to the lack of inter-element continuity requirements, e.g., meshes that contain several different types of elements and/or hanging nodes. Both those features are crucial in many applications. 
In the context of multiscale methods the discontinuous formulation allows for more flexibility in the construction of the basis function e.g., allowing more general boundary conditions \cite{EGM12}.
%
%
Although the error analysis presented in this work is restricted to regular simplicial or quadrilateral/hexahedral meshes, we stress that all the results appear to be extendable for the case of irregular meshes (i.e., meshes containing hanging nodes). We refrained from presenting these extensions here for simplicity of the current presentation. Under these assumptions, we provide a complete a priori error analysis of this method including errors caused by the approximation of basis functions. 

In this dG multiscale method and in previous related methods \cite{MP11,EGM12}, the accuracy is ensured by enlarging the support of basis functions appropriately. Hence, supports of basis functions overlap and the communication is no longer restricted to neighboring elements but is present also between elements at a certain distance. We will prove that resulting overhead is acceptable in the sense that it scales only logarithmically in the mesh size. 

In order to retain the dG-typical structure of the stiffness matrix, we discuss the possibility of localizing the multiscale basis functions to single elements. Instead of having $\mc{O}(1)$ basis functions per element with $\mc{O}(H|\log(H^{-1})|)$ support, we would then have $\mathcal{O}(|\log(H^{-1})|)$ basis functions per element with element support. The element-wise application of a singular value decomposition easily prevents ill-conditioning of the element stiffness matrices, while simultaneously achieving further compression of the multiscale basis. 

The outline of the paper is as follows. In Section~\ref{s:discretesetting}, we recall the dG finite element method. Section~\ref{s:multiscale} defines our multiscale method, which is then analyzed in Section~\ref{s:apriori}. Section~\ref{s:numexp} presents numerical experiments confirming the theoretical developments. 

Throughout this paper, standard notation on Lebesgue and Sobolev spaces is employed. Let $0\leq C<\infty$ be any generic constant that neither depends on the mesh size nor the diffusion matrix $A$; $a\lesssim b$ abbreviates an inequality $a\leq C\,b$ and $a\approx b$ abbreviates $a\lesssim b\lesssim a$. Also, let the constant $\ba$ depend on the minimum and maximum bound bound ($\amin$ and $\amax$) of the diffusion matrix $\mscale$ \eqref{eq:ba}.

\section{Fine scale discretization}\label{s:discretesetting}

\subsection{Finite element meshes and spaces}

Let $\T$ denote a subdivision of $\Omega$ into (closed) regular simplices or into quadrilaterals (for $d=2$) or hexahedra (for $d=3$), i.e., $\bar{\Omega}=\cup_{T\in\T} T$. We assume that $\T$ is regular in the sense that any two elements are either disjoint or share exactly one face or share exactly one edge or share exactly one vertex. 

Let $\E$ denote the set of edges (or faces for $d=3$) of $\T$; $\E(\Omega)$ denotes the set of interior edges, $\E(\Gamma)$, $\E(\Gamma_D)$ and $\E(\Gamma_N)$) refer to the set of edges on the boundary of $\Omega$, on the Dirichlet and on the Neumann boundary, respectively.
Let $\hat{T}$, denote the reference simplex or (hyper)cube and let $\mathcal{P}_p(\hat{T})$ and $\mathcal{Q}_p(\hat{T})$ denote the spaces of polynomials of degree less than or equal to $p$ in all or on each variable, respectively. Then, we define the set of piecewise polynomials
\begin{equation*}
P_p(\T) := \{v:\Omega\rightarrow \mathbb{R}\;\vert \;\forall T\in\T,v\vert_T\circ F_{T}\in\mathcal{R}_p(\hat{T})\},
\end{equation*}
with $\mathcal{R}_p\in\{\mathcal{P}_p,\mathcal{Q}_p\}$, where $F_T:\hat{T}\to T$, $T\in\T$ is a family of element maps. Let also $\Pi_p(\T):L^2(\Omega)\rightarrow P_p(\T)$ denote the $L^2$-projection onto $\T$-piecewise polynomial functions of order $p$. In particular, we have ${(\Pi_0(\T) f)\vert_T=|T|^{-1}\int_T f\dx}$, $T\in\T$ for all $f\in L^2(\Omega)$. 
Note that $v\in P_p(\T)$ does not necessarily belong to $H^1(\Omega)$. The $\T$-piecewise gradient $\nabla_\T v$, with $(\nabla_\T v)\vert_T=\nabla (v\vert_T)$ for all $T\in\T$, is well-defined and $\nabla_\T v\in (P_{p-1}(\T))^d$.

For any interior edge/face $e\in\E(\Omega)$ there are two adjacent elements $T^-$ and $T^+$ with $e=\partial T^-\cap\partial T^+$. We define $\nu$ to be the normal vector of $e$ that points from $T^-$ to $T^+$. For boundary edges/faces $e\in\mathcal{E}(\Gamma)$ let $\nu$ be the outward unit normal vector of $\Omega$.

Define the jump of $v\in P_k(\T)$ across $e\in\E(\Omega)$ by $[v]:=v\vert_{T^-}-v\vert_{T^+}$ and define $[v]:=v\vert_{e}$ for $e\in\E(\Gamma)$. The average of $v\in P_p(\T)$ across $e\in\E(\Omega)$ is defined by $\{ v\}:=(v\vert_{T^-}+v\vert_{T^+})/2$ and for boundary edges $e\in\E(\Gamma)$ by $\{ v\}:=v\vert_e$.

In the remaining part of this work, we consider two different meshes: a coarse mesh $\T_H$ and a fine mesh $\T_h$, with respective definitions for the edges/faces $\E_H$ and $\E_h$. We denote the $\T_H$-piecewise gradient by $\nabla_Hv:=\nabla_{\T_H}v$ and, respectively, $\nabla_hv:=\nabla_{\T_h}v$ for the $\T_h$-piecewise gradient . We assume that the fine mesh $\T_h$ is the result of one or more refinements of the coarse mesh $\T_H$. The subscripts $h,H$ refer to the corresponding mesh sizes; in particular, we have $H\in P_0(\T_H)$ with $H\vert_T=\diam(T)=:H_T$ for all $T\in\T_H$, $H_e=\diam{e}$, for all $e\in\E_H$, $h\in P_0(\T_h)$, with $h\vert_T=\diam(T)=:h_T$ for all $T\in\T_h$, and $h_e=\diam{e}$ for all $e\in\E_h$. Obviously, $h\leq H$.

\subsection{Discretization by the symmetric interior penalty method}
We consider the symmetric interior penalty method (SIP) discontinuous Galerkin method \cite{MR0440955,MR664882,MR2290408}. We seek an approximation in the space $\Vh:=P_1(\T_h)$. 
 Given some positive penalty parameter $\sigma>0$, we define the symmetric bilinear form $a_h:\Vh\times \Vh\rightarrow\mathbb{R}$ by
\begin{equation}\label{e:bilindg}
  \begin{aligned}
  a_h(u,v) := &(A\nabla_h u,\nabla_h v)_{L^2(\Omega)} - \sum_{e\in\E_h(\Omega)\cup\E_h(\Gamma_D)}\Big((\{\nu\cdot\mscale\nabla u\},[v])_{L^2(e)} \\
  &  +(\{\nu\cdot\mscale\nabla v\},[u])_{L^2(e)}-\frac{\sigma}{h_e}([u],[v])_{L^2(e)}\Big).
  \end{aligned}
\end{equation}
The jump-seminorm associated with the space $\Vh$, is defined by 
\begin{equation}\label{e:dgjumpnorm}
|\bullet|_{h}^2:=\sum_{e\in\E_h(\Omega)\cup\E_h(\Gamma_D)} \frac{\sigma}{h_e}\|[\bullet]\|_{L^2(e)}^2,
\end{equation}
while the energy norm in $\Vh$ is then given by
\begin{equation}\label{e:dgnorm}
\Enormh{\bullet}:=(\|A^{1/2}\nabla_h\bullet\|_{L^2(\Omega)}^2+|\bullet|_{h}^2)^{1/2}.
\end{equation}
If the penalty parameter is chosen sufficiently large, the dG bilinear form \eqref{e:bilindg} is coercive and bounded with respect to the energy norm \eqref{e:dgnorm}. Hence, there exists a (unique) dG approximation $u_h\in\Vh$, satisfying
\begin{equation}\label{e:DGweak}
a_h(u_h,v) = F(v) \quad\text{ for all }v\in \Vh.
\end{equation}

We assume that  \eqref{e:DGweak} is computationally intractable for practical problems, so we shall never seek to solve for $u_h$ directly. Instead, $u_h$ will serve as a reference solution to compare our coarse grid multiscale dG approximation with. The underlying assumption is that the mesh $\T_h$ is chosen sufficiently fine so that $u_h$ is sufficiently accurate. The aim of this work is to devise and analyse a multiscale dG discretization with coarse scale $H$, in such a way that the accuracy of $u_h$ is preserved up to an $\mc{O}(H)$ perturbation independent of the variation of the coefficient $A$.

\section{Discontinuous Galerkin multiscale method}\label{s:multiscale}

As mentioned above, the choice of the reference mesh $\T_h$ is not directly related to the desired accuracy, but is instead strongly affected by the roughness and variation of the coefficient $A$. The corresponding coarse mesh $\T_H$, with mesh width function $H\geq h$, is assume to be completely independent of $A$. To encapsulate the fine scale information in the coarse mesh, we shall design coarse generalized finite element spaces based on $\T_H$. 

\subsection{Multiscale decompositions}
We introduce a scale splitting for the space $\Vh$. To this end, let $\Lp:=\Pi_1(\T_H)$ and define $\VH:=\Lp\Vh=P_1(\T_H)$ and
\begin{equation*}
\Vfh:=(1-\Lp)\Vh=\{v\in\Vh\;\vert\;\Lp v= 0\}.
\end{equation*}
\begin{lemma}[$L^2$-orthogonal multiscale decomposition]\label{l:odl2}
The decomposition
$$\Vh=\VH\oplus \Vfh,$$
is orthogonal in $L^2(\Omega)$.
\end{lemma}
\begin{proof}
The proof is immediate, as any $v\in\Vh$ can be
decomposed uniquely into a coarse finite element function $v_H:=\Lp v\in \VH$ and a (possibly highly oscillatory) remainder $v^{\operatorname*{f}}:=(1-\Lp)v\in \Vfh$, with $\|v\|_{L^2(\Omega)}^2=\|v_H\|_{L^2(\Omega)}^2+\|v^{\operatorname*{f}}\|_{L^2(\Omega)}^2$.
\end{proof}
We now orthogonalize the above splitting with respect to the dG scalar product $a_h$; we keep 
the space of fine scale oscillations $\Vfh$ and simply replace $\VH$ with the orthogonal complement of $\Vfh$ in $\Vh$. We define the fine scale projection $\F:\Vh\rightarrow \Vfh$ by 
\begin{equation}\label{eq:multiscale-map}
  a_h(\F v,w) = a_h(v,w)\quad \text{for all }w\in\Vfh.
\end{equation}
Using the fine scale projection, we can define the coarse scale approximation space by
\begin{equation*}
 \VmshH:= (1-\F)\VH.
\end{equation*}

\begin{lemma}[$a_h$-orthogonal multiscale decomposition]\label{l:odh1}
The decomposition
$$ \Vh=\VmshH\oplus \Vfh, $$
is orthogonal with respect to $a_h$, i.e., any function $v$ in $\Vh$ can be
decomposed uniquely into some function $v^{ms}_H\in \VmshH$ plus $v^{\operatorname*{f}}\in \Vfh$ with $\Enormh{v}^2\approx\Enormh{v^{ms}_H}^2+\Enormh{v^{\operatorname*{f}}}^2$.
The functions $v^{ms}_H\in \VmshH$ and $v^{\operatorname*{f}}\in \Vfh$ are the Galerkin projections of $v\in \V_h$ onto the subspaces $\VmshH$ and $\Vfh$, i.e.,
\begin{align*}
 a_h(v^{ms}_H,w)&=a_h(v,w)=F(w)\quad \text{for all }w\in\VmshH,\\
 a_h(v^{\operatorname*{f}},w)&=a_h(v,w)=F(w)\quad \text{for all }w\in\Vfh.
\end{align*}
\end{lemma}
The unique Galerkin approximation $\umshH\in \VmshH$ of $u\in \V$ solves
\begin{equation}\label{eq:GMM}
  \Bh{\umshH}{v} = F(v)\quad \text{for all }v\in\VmshH.
\end{equation}
We shall see in the error analysis (cf. Theorem~\ref{thm:umshH}) that the orthogonality yields error estimates for the Galerkin approximation $\umshH\in \VmshH$ of \eqref{eq:GMM} that are independent of the regularity of the solution and of the diffusion coefficient $A$. 
However, the space $\VmshH$ is not suitable for practical computations as a local basis for this space is not easily available. 
Indeed, given a basis of $\VH$, e.g., the element-wise Lagrange basis functions $\{\lambda_{T,j}\;|\;T\in\T_H,\;j=1,\ldots,r\}$ where $r=(1+d)$ for regular simplices or $r=2^d$ for quadrilaterals/hexahedra.
 The space $\VmshH$ may be spanned by the
corrected basis functions $(1-\F)\lambda_{T,j}$, $T\in\T_H,\;j=1,\ldots,r$. 
Although $\lambda_{T,j}$ has local support $\supp{\lambda_{T,j}}=T$, its corrected version $(1-\F)\lambda_{T,j}$ may have global support in $\Omega$, as \eqref{eq:multiscale-map} is a variational problem on the whole domain $\Omega$. Fortunately, as we shall prove later, the corrector functions $\phi_{T,j}$ decay quickly away from $T$ (cf. previous numerical results in \cite{EGM12} and a similar observation for the corresponding conforming version of the method \cite{MP11}). This decay motivates the local approximation of the corrector functions, at the expense of introducing small perturbations in the method's accuracy. 

\subsection{Discontinuous Galerkin multiscale method}
The localized approximations of the corrector functions are supported on element patches in the coarse mesh $\T_H$.
\begin{definition} For all $T\in\T_H $, define element patches with size $L$ as
  \begin{equation}\notag
    \begin{aligned}
      \omega^1_T & := \text{int}(T), \\
      \omega^{L}_T & := \text{int}(\cup\{T'\in \T_H\;\vert\; T'\cap\bar{\omega}_T^{L-1}\neq\emptyset\}),\quad k=1,2,\ldots .
    \end{aligned}
  \end{equation}
  We refer to \fgref{fig:mesh} for an illustration.
\end{definition}
\begin{figure}\label{fig:mesh}
  \centering
  \includegraphics[scale=0.7]{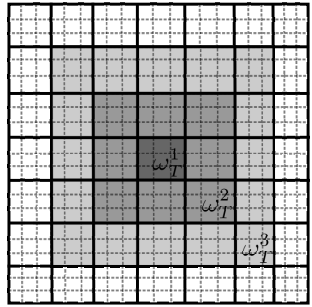}
  \caption{Example of a one layer patch $\omega^1_T$, two layer patch $\omega^2_T$, and a three layer patch $\omega^3_T$, on a quadrilateral mesh.}
\end{figure}
We introduce a new discretization parameter $L>0\in\mathbb{N}$ and define localized corrector functions $\phi^L_{T,j}\in\Vfh(\omega^L_T):=\{v\in\V^{\operatorname*{f}} \mid v|_{\Omega\setminus\omega^L_T}=0\}$ by
\begin{equation}\label{eq:corrector}
  a_h(\phi^L_{T,j},w) = a_h(\lambda_{T,j},w)\quad \text{for all }w\in\Vfh(\omega^L_T).
\end{equation}
Further, we define the multiscale approximation space
\begin{equation*}
 \VmshHL=\text{span}\{\lambda_{T,j}-\phi^L_{T,j}\;\vert\; T\in\T_H,\;j=1,\ldots,r\},
\end{equation*}

The dG multiscale method seeks $\umshHL\in\VmshHL$ such that
\begin{equation}\label{eq:DGMM}
  \Bh{\umshHL}{v} = F(v)\quad \text{for all }v\in\VmshHL.
\end{equation}
Since $\VmshHL\subset \Vh$, this method is a Galerkin method in the Hilbert space $\Vh$ (with scalar product $a_h$) and, hence, inherits well-posedness from the reference discretization \eqref{e:DGweak}.

Moreover, the proposed basis $\{\lambda_{T,j}-\phi^L_{T,j}\;\vert\; T\in\T_H,\;j=1,\ldots,r\}$ is stable with respect to the fine scale parameter $h$, as we shall see in Lemma~\ref{l:stabbasis} below.

\subsection{Compressed discontinuous Galerkin multiscale method}\label{ss:compressed}
The basis functions in the above multiscale method have enlarged supports (element patches) when compared with standard dG methods (elements). We can decompose the corrector functions into its element contributions
\begin{equation*}
 \phi^L_{T,j}=\sum_{T'\in\T_H:T'\subset\omega_T^L} \phi^L_{T,j}\chi_{T'},
\end{equation*}
where $\chi_{T'}$ is the indicator function of the element $T'\in\T_H$. 

This motivates the coarse approximation space
\begin{align*}
 \WmshHL=&\text{span}\bigl(\{\lambda_{T,j}\;\vert T\in\T_H,\;j=1,\ldots,r\}\\&\cup
\{\phi^L_{T,j}\chi_{T'}\;\vert T,T'\in\T_H,\;T'\subset\omega_T^L,\;j=1,\ldots,r\},
\end{align*}
This space offers the advantage of a known basis with element-wise support which leads to improved (localized) connectivity in the corresponding stiffness matrix. This is at the expense of a slight increase in the dimension of the space
$$\dime{\WmshHL} \approx L^d\dime{\VmshHL}.$$

The corresponding localized dG multiscale method seeks $\wmshHL\in\WmshHL$ such that
\begin{equation}\label{eq:DGMMc}
  \Bh{\wmshHL}{v} = F(v)\quad \text{for all }v\in\WmshHL.
\end{equation}

Since $\VmshHL\subset\WmshHL\subset \Vh$, Galerkin orthogonality yields
\begin{equation}\label{e:DGMMcbest}
\Enormh{u_h-\wmshHL}\leq\Enormh{u_h-\umshHL},
\end{equation}
i.e., the new localized version \eqref{eq:DGMMc} is never worse than the previous multiscale approximation in terms of accuracy. However, it may lead to very ill-conditioned element stiffness matrices (cf. Lemma~\ref{lem:decay_layers} which shows that $\phi^L_{T,j}\xi_{T'}$ may be very small if the distance between $T$ and $T'$ relative to their sizes is large). 

To circumvent ill-conditioning, one may choose a reduced local approximation space on the basis of a singular value decomposition. Since the dimension of the local approximation space is small (at most proportional to $L^d$), the cost for this additional preprocessing step is comparable with the cost for the solution of the local problems for the corrector functions.


To determine an acceptable level of truncation of the localized basis functions, we can use the the a posteriori error estimator contribution of the local problem from \cite{EGM12}, which is an estimation of the local fine scale error. This procedure may additionally lead to large reduction of the dimension of the local approximation spaces. 

\section{Error analysis}\label{s:apriori}
We present an a priori error analysis for the proposed multiscale method \eqref{eq:DGMM}. In view of \eqref{e:DGMMcbest}, this analysis applies immediately to the modified versions presented in Section  \ref{ss:compressed}. The error analysis will be split into a number of steps.
First, in Section \ref{projection_properties}, we present some properties of the coarse scale projection operator $\Lp$.
In Section \ref{Mdiscretization}, an error bound for dG multiscale method $\umshH$ from \eqref{eq:GMM} (\thmref{thm:umshH}) is shown, whereby the corrected basis functions are solved globally. Results for the decay of the localized corrected basis function (\lemref{lem:decay_layers} and \lemref{lem:discrete_finescale_decay}) are shown, along with an error bound for the dG multiscale method $\umshHL$ from \eqref{eq:DGMM} (\thmref{thm:discrete}) ,where the corrected basis functions are solved locally on element patches. Finally, in Section \ref{ss:quantity}, we show an error bound given a quantity of interest (\thmref{thm:qoi}), leading to an error bound in $L^2$-norm (\corref{cor:L2bound}).

We shall make use of the following (semi)norms.
The jump-seminorm and energy norms, associated with the coarse space $\VH$, are defined by
\begin{equation}\notag
  \begin{aligned}
    |\bullet|_{H}^2&:=\sum_{e\in\E_H(\Omega)\cup\E_H(\Gamma_D)} \frac{\sigma}{H_e}\|[\bullet]\|_{L^2(e)}^2, \\
\EnormH{\bullet}&:=(\|A^{1/2}\nabla_H\bullet\|_{L^2(\Omega)}^2+|\bullet|_H^2)^{1/2},
\end{aligned}
\end{equation}
respectively, along with a localized version of the local jump and energy norms \eqref{e:dgjumpnorm} and \eqref{e:dgnorm} on a patch $\omega\subseteq\Omega$, where $\omega$ is aligned with the mesh $\T_h$, given by
\begin{equation}\notag
\begin{aligned}
|\bullet|_{h,\omega}^2&:=\sum_{\substack{e\in\E(\Omega)\cup\E(\Gamma_D):\\e\cap\bar{\omega}\neq0}} \frac{\sigma}{h_e}\|[\bullet]\|_{L^2(e)}^2, \\
\Enormhw{\bullet}{\omega}&:=(\|A^{1/2}\nabla_h\bullet\|_{L^2(\omega)}^2+|\bullet|_{h,\omega}^2)^{1/2}.
\end{aligned}
\end{equation}
The shape-regularity assumptions $h_T\approx h_e$ for all $e\in\partial T:T\in\T_h$ and $H_T\approx H_e$ for all $T\in\partial T:T\in\T_H$ will also be used. 


\subsection{Properties of the coarse scale projection operator $\Lp$}\label{projection_properties}
The following lemma gives stability and approximation properties of the operator $\Lp$.
\begin{lemma}\label{lem:Lp_approx}
  For any $v\in\Vh$, the estimate
  \begin{align}\notag
    H^{-1}\|v-\Lp v\|_{L^2(T)}  \lesssim\amin^{-1/2}\Enormhw{v}{T},
  \end{align}
  is satisfied for all $T\in\T_H$. Moreover, it holds
  \begin{align}\notag
    \amax^{-1/2}\EnormH{\Lp v}+\|H^{-1}(v-\Lp v)\|_{L^2(\Omega)}  \lesssim\amin^{-1/2}\Enormh{v}.
  \end{align}
\end{lemma}
\begin{proof}
  Theorem 2.2 in \cite{Karakashian2003}, implies that for each $v\in \Vh$, there exists a bounded linear operator $\Ich:\Vh\rightarrow\Vh\cap H^1(\Omega)$ such that
\begin{equation}\label{e:DGGudiN3_h}
  \amax^{-1/2}\|\mscale^{1/2}\nabla_H(v-\Ich v)\|_{L^2(\Omega)} + \|h^{-1}(v-\Ich v)\|_{L^2(\Omega)}\lesssim \amin^{-1/2} |v|_h.
\end{equation}  
Split $v=v^{\operatorname*{c}}+v^{\operatorname*{d}}\in\V_h$ into a conforming, $v^{\operatorname*{c}}=\Ich v$, and non-conforming, $v^{\operatorname*{d}}=v-\Ich v$, part. We obtain
  \begin{equation}\label{eq:approxLP}
    \begin{aligned}
      H^{-1}\|v-\Lp v\|_{L^2(T)} &\leq H^{-1}(\|v^{\operatorname*{c}}-\Lp v^{\operatorname*{c}}\|_{L^2(T)}+\|v^{\operatorname*{d}}-\Lp v^{\operatorname*{d}}\|_{L^2(T)}) \\
      & \lesssim \|\nabla_h v\| +\|\nabla_h(v-v^{\operatorname*{c}})\|_{L^2(T)} + H^{-1}\|v^{\operatorname*{d}}\|_{L^2(T)}) \\
      & \lesssim \amin^{-1/2}\Enormhw{v}{T}
    \end{aligned}      
  \end{equation}
  using the triangle inequality, stability of the $L^2$-projection, and \eqref{e:DGGudiN3_h}.
  Also,
  \begin{equation}\notag
    \begin{aligned}
      \EnormH{\Lp v}^2 =& \sum_{T\in\T_H}\|\sqrt{\mscale}\nabla (\Lp v -\Pi_0(\T_H)v)\|^2_{L^2(T)}+\sum_{e\in\Gamma}\frac{\sigma}{H}\|[v_c-\Lp v]\|^2_{L^2(e)} \\
       \lesssim& \sum_{T\in\T_H}\amax\left(\frac{1}{H^2}\|v-\Pi_0(\T_H)v\|^2_{L^2(T)}+\frac{1}{H^2}\|v_c-\Lp v\|^2_{L^2(T)}\right)\\
      \lesssim & \ba^2\Enormh{v}^2,
    \end{aligned}
  \end{equation}
using the triangle inequality, \eqref{e:DGGudiN3_h}, and \eqref{eq:approxLP} which concludes the proof.
\end{proof}

The operator $\Lp$ is surjective. The next lemma shows that, given some $v_H\in \VH$ in the image of $\Lp$ there exists a $H^1$-conforming pre-image $v\in \Lp^{-1}v_H\subset \Vh$ with comparable support.
\begin{lemma}\label{lem:property_Lp}
  For each $v_H\in\VH$, there exists a $v\in \Vh\cap H^1(\Omega)$ such that $\Lp v = v_H$, $\Enormh{v}\lesssim \ba\EnormH{v_H}$, and $\text{supp}(v)\subseteq\text{supp}(v_H)$.
\end{lemma}
\begin{proof}
Using \eqref{e:DGGudiN3_h} but on space $\VH$ gives, for each $v\in \VH$, there exists a bounded linear operator $\IcH:\VH\rightarrow\VH\cap H^1(\Omega)$ such that
\begin{equation}\label{e:DGGudiN3}
\amax^{-1/2}\|\mscale^{1/2}\nabla_H(v-\IcH v)\|_{L^2(\Omega)} + \|H^{-1}(v-\IcH v)\|_{L^2(\Omega)}\lesssim \amin^{-1/2} |v|_H.
\end{equation}
We define
  \begin{equation}\notag
    v := \IcH v_H+\hspace{-3ex}\sum_{T\in\T_H,\;j=1,\ldots,r}\hspace{-3ex}\left( v_H(x_j)-\IcH v_H(x_j) \right)\theta_{T,j},
  \end{equation}
  where $\theta_{T,j}\in \Vh\cap H^1_0(T)$ are bubble functions, supported on each element $T$, with $\Lp \theta_{T,j} = \lambda_{T,j}$ and $\Enormh{\theta_{T,j}}\lesssim\amax(T)H^{d-2}$. Observe that $\text{supp}(v)\subseteq\text{supp}(v_H)$.  The interpolation property follows from
  \begin{equation}
    \begin{aligned}\notag
      \Lp v &= \IcH v_H+\Lp\hspace{-3ex}\sum_{T\in\T_H,\;j=1,\ldots,r}\hspace{-3ex}\left( v_H(x_j)-\IcH v_H(x_j) \right)\theta_j, \\
      & =  \IcH v_H   +  v_H -  \IcH v_H = v_H.
    \end{aligned}
  \end{equation}

To prove stability, we estimate $\Enormh{v}$ as follows:
  \begin{equation}
    \begin{aligned}\notag
      \Enormh{v}^2 & \leq \|\mscale^{1/2}\nabla\IcH v_H\|^2_{L^2(\Omega)}+\hspace{-3ex}\sum_{T\in\T_H,\;j=1,\ldots,r}\hspace{-3ex}\left| v_H(x_j)-\IcH v_H(x_j) \right|^2\Enormh{\theta_j}^2\\
      &\lesssim \|\mscale^{1/2}\nabla_H\IcH v_H\|^2_{L^2(\Omega)}+\amax\|H^{-1}( v_H-\IcH v_H) \|^2_{L^2(\Omega)}\\
      &\lesssim \ba^2 \EnormH{v_H}^2
    \end{aligned}
  \end{equation}
  using the inverse estimate $\|v\|_{L^\infty(T)}\leq H^{-d/2}\|v\|_{L^2(T)}$ for all $v\in\VH$, and the estimate \eqref{e:DGGudiN3}.
\end{proof}
\begin{rem}
  Note that $\theta_{T,j}\in\Vh\cap H^1_0(T)$ for all $T\in\T_H$ (fulfilling the conditions in \lemref{lem:property_Lp}) can be constructed using two (or more) refinements of the coarse scale parameter $H$. We can let $\theta_{T,j}\in\V_{h'}\cap H^1_0(T)$ where $\V_{h'}\subset\V_h$ and $h\leq h'\leq 2^{-2}H$. This does not put a big restriction on $h$ since the mesh $\T_h$ is assumed to be sufficiently fine to resolve the variation in the coefficient $\mscale$, while the parameter $H$ does not need to resolve $\mscale$.
\end{rem}

The following lemma says that the corrected basis is stable with respect to the fine scale parameter $h$ in the energy norm \eqref{e:dgnorm}, this is not a trivial result since the basis function $\{\lambda_{T,j}|T\in\T_H,\;j=1,\ldots,r\}$ are discontinuous.
\begin{lemma}[Stability of the corrected basis functions]\label{l:stabbasis}
  For all, $T\in\T_H,\;j=1,\ldots,r$ and $L>0\in\mathbb{N}$, the estimate
  \begin{equation}\notag
    \Enormh{\lambda_{T,j}-\phi^L_{T,j}} \lesssim \ba\EnormH{\lambda_{T,j}},
  \end{equation}
  is satisfied, independently of the fine scale parameter $h$.
\end{lemma}
\begin{proof} For any $T\in\T_H,\;j=1,\ldots,r$, by \lemref{lem:property_Lp} there exists a $b$ such that $v=\lambda_{T,j}-b\in\Vfhw{\omega_T^L}$, and $\Enormh{b}\lesssim\ba\EnormH{\lambda_{T,j}}$. We have
  \begin{equation}\notag
    \begin{aligned}
    \Enormh{\lambda_{T,j}-\phi^L_{T,j}}^2 &\lesssim \Bh{\lambda_{T,j}-\phi^L_{T,j}}{\lambda_{T,j}-\phi^L_{T,j}}=\Bh{\lambda_{T,j}-\phi^L_{T,j}}{\lambda_{T,j}-v}, \\
&\lesssim \Bh{\lambda_{T,j}-\phi^L_{T,j}}{b}\lesssim \ba\Enormh{\lambda_{T,j}-\phi^L_{T,j}}\EnormH{\lambda_{T,j}},
    \end{aligned}
  \end{equation}
  which concludes the proof.
\end{proof}

\subsection{A priori estimates}\label{Mdiscretization}
The following theorem gives an error bound for the idealized dG multiscale method, whereby the correctors for the basis are solved globally.
\begin{theorem}\label{thm:umshH}
  Let $u_h\in\Vh$ solve \eqref{e:DGweak} and $\umshH\in\VmshH$ solve \eqref{eq:DGMM}, then the estimate
  \begin{equation}\notag
    \Enormh{u_h-\umshH} \leq C_1\amin^{-1/2}|| H(f-\Lp f ) ||_{L^2(\Omega)},
  \end{equation}
  is satisfied, where $C_1$ neither depends on the mesh ($h$ or $H$) size nor the diffusion matrix $\mscale$.
\end{theorem}
\begin{proof}
  Let $e := u_h-\umshH=u_f\in \Vfh$, then
  \begin{equation}\notag
    \begin{aligned}
    \Enormh{e}^2 &\lesssim \Bh{e}{e}= (f,e)_{L^2(\Omega)} = (f-\Lp f,e-\Lp e)_{L^2(\Omega)} \\
    & \leq ||H(f-\Lp f)||_{L^2(\Omega)}||H^{-1}(e-\Lp e)||^2_{L^2(\Omega)}\\
    & \lesssim \frac{1}{\sqrt{\amin}}||H (f-\Lp f)||_{L^2(T)}\Enormh{e},
    \end{aligned}
  \end{equation}
  using \lemref{l:odh1}, \lemref{l:odl2}, Cauchy-Schwarz inequality, and \lemref{lem:Lp_approx}, respectively.
\end{proof}

\begin{definition}
  The cut off functions $\zeta_T^{d,D}\in P_0(\T_h)$ are defined by the conditions
  \begin{equation}\notag
    \begin{aligned}
      \zeta_T^{d,D}|_{\omega_T^d} &= 1, \\
      \zeta_T^{d,D}|_{\Omega\setminus\omega^D_T} &= 0, \\
      \|[\zeta_T^{d,D}]\|_{L^\infty(T)}& \lesssim \frac{h_T}{(D-d)H_T}\quad\text{for all }T\in\T_H,
    \end{aligned}
  \end{equation}
  and that $\zeta_T^{d,D}$ is constant on the boundary $\partial(\omega^D_T\setminus\omega^d_T)$.
\end{definition}

The next lemma shows the exponential decay in the corrected basis, which is a key result in the analysis.
\begin{lemma}\label{lem:decay_layers}
   For all $T\in\T_H,\;j=1,\ldots,r$, the estimate
  \begin{equation}\notag
    \Enormh{(\lambda_j-\phi_{T,j})-(\lambda_j-\phi^L_{T,j})} =\Enormh{\phi_{T,j}-\phi^L_{T,j}} \leq  C_3\gamma^L\Enormh{\phi_{T,j}-\lambda},
  \end{equation}
  is satisfied, with $C_3=C\ba^3$, $0<\gamma<1$ given by $\gamma := (\frac{ C_2}{\ell})^{\frac{k-1}{2\ell}}$, $C_2=C'\ba^2$, and $L = k\ell$, $k,\ell\geq2\in\mathbb{N}$, noting that $C$ and $C'$ are positive constants that are independent of the mesh ($h$ or $H$), of the patch size $L$, and of the diffusion matrix $\mscale$.
\end{lemma} 
\begin{proof}
  Define $e:=\phi_{T,j}-\phi^L_{T,j}=\phi_{T,j}-\phi^{\ell k}_{T,j}$. We have
  \begin{equation}\label{eq:decay1}
    \begin{aligned}
      \Enormh{e}^2 &\lesssim \Bh{e}{\phi_{T,j}-\phi^{\ell k}_{T,j}} = \Bh{e}{\phi_{T,j}-v} \lesssim \Enormh{e}\cdot\Enormh{\phi_{T,j}-v},
    \end{aligned}
  \end{equation}
  for $v\in\Vfhw{\omega^{\ell k}_T}$.
  Let $\zeta:=\zeta_{T}^{\ell k-1,\ell k}$; then by \lemref{lem:property_Lp} there exists a $b$ such that $v=\zeta\phi_{T,j} - b\in\Vfhw{\omega^{\ell k}_T}$,  $\Lp b = \Lp\zeta\phi_{T,j}$, $\Enormh{b}\lesssim\ba\EnormH{\Lp\zeta\phi_{T,j}}$, and $\text{supp}(b)\subseteq\text{supp}(\Lp\zeta\phi_{T,j})$. Then, we have 
  \begin{equation}\label{eq:decay2}
    \begin{aligned}
      \Enormh{\phi_{T,j}-v} &= \Enormh{\phi_{T,j} - (\zeta\phi_{T,j} - b)} \\ 
      &\leq\Enormh{\phi_{T,j} - \zeta\phi_{T,j}} + \Enormh{b} \\
      &\lesssim \Enormh{\phi_{T,j} - \zeta\phi_{T,j}} + \ba\EnormH{\Lp (\zeta\phi_{T,j}-\phi_{T,j})}\\
      &\lesssim \ba^2\Enormh{\phi_{T,j} - \zeta\phi_{T,j}}.
    \end{aligned}
  \end{equation}
 Furthermore, using the properties of $\zeta$ we have
  \begin{align}\label{eq:decay3}
    \|\sqrt{\mscale}\nabla_h(1-\zeta)\phi_{T,j}\|_{L^2(\Omega)}, 
\leq \|\sqrt{\mscale} \nabla_h\phi_{T,j} \|_{L^2(\Omega\setminus\omega^{\ell k-1}_T)},
  \end{align}
  and
  \begin{equation}\label{eq:decay4}
    \begin{aligned}
      & |(1-\zeta)\phi_{T,j} |^2_{h} = \sum_{e\in\E(\Omega)\cup\E(\Gamma_D)} \frac{\sigma}{h_e}\|[(1-\zeta)\phi_{T,j}]\|^2_{L^2(e)}  \\
      &\quad\leq \sum_{e\in\E(\Omega)\cup\E(\Gamma_D)}\frac{\sigma}{h_e}\left(\|\{1-\zeta\}[\phi_{T,j}]\|^2_{L^2(e)} + \|\{\phi_{T,j}\}[1-\zeta]\|^2_{L^2(e)} \right)\\
      &\quad\leq \sum_{\substack{e\in\E(\Omega)\cup\E(\Gamma_D):\\e\cap\Omega\setminus\omega^{\ell k-1}_T\neq0}}\left(\frac{\sigma}{h_e}\|[\phi_{T,j}]\|^2_{L^2(e)} +  \frac{\sigma h_T}{h_eH_T^2}\|\{\phi_{T,j}\}\|^2_{L^2(e)}\right) \\
      &\quad\leq \sum_{\substack{e\in\E(\Omega)\cup\E(\Gamma_D):\\e\cap\Omega\setminus\omega^{\ell k-1}_T\neq0}}\frac{\sigma}{h_e}\|[\phi_{T,j}]\|^2_{L^2(e)} + \frac{\sigma}{H_T^2}\|\phi_{T,j}-\Lp\phi_{T,j}\|^2_{L^2(\Omega\setminus\omega^{\ell k-1}_T)} \\
      &\quad\lesssim \ba^2\Enormhw{\phi_{T,j}}{\Omega\setminus\omega^{\ell k-1}_T},
    \end{aligned}
  \end{equation}
  using a trace inequality and \lemref{lem:Lp_approx}, respectively. Combining \eqref{eq:decay1}, \eqref{eq:decay2}, \eqref{eq:decay3}, and \eqref{eq:decay4} yields
  \begin{align}\label{eq:decay_11}
    \Enormh{\phi_{T,j} - \zeta\phi_{T,j}} \lesssim  \ba^3\Enormhw{\phi_{T,j}}{\Omega\setminus\omega^{\ell k-1}_T}.
  \end{align}
To simplify notation, let $m:=\ell (k-1)-1$ and $M:=\ell k -1$.
For $\eta_T:=1-\zeta_T^{m,M}$, we obtain
  \begin{equation}\label{eq:decay5}
    \Enormhw{\phi_{T,j}}{\Omega\setminus\omega^{M}_T}^2 \leq \Enormh{\eta_T\phi_{T,j}}^2\lesssim \Bh{\eta_T\phi_{T,j}}{\eta_T\phi_{T,j}},
  \end{equation}
where
  \begin{equation}\label{eq:decaycutofbilinear}
    \begin{aligned}
      &\Bh{\eta_T \phi_{T,j}}{\eta_T \phi_{T,j}} = (\mscale\nabla_h \eta_T \phi_{T,j},\nabla_h \eta_T \phi_{T,j})_{L^2(\Omega)} \\ &\qquad + \sum_{e\in \E(\Omega)\cup\E(\Gamma_D)}\left(-2(\{\nu\cdot\mscale\nabla \eta_T\phi_{T,j}\},[\eta_T\phi_{T,j}])+\frac{\sigma}{h_e}([\eta_T \phi_{T,j}],[\eta_T \phi_{T,j}])\right).
    \end{aligned}
  \end{equation}
  For the first term on the right-hand side of \eqref{eq:decaycutofbilinear}, we have
  \begin{equation}\label{eq:decay6}
    (\mscale\nabla_h \eta_T \phi_{T,j},\nabla_h \eta_T \phi_{T,j})_{L^2(\Omega)} = (\mscale\nabla_h \phi_{T,j},\nabla_h \eta_T^2 \phi_{T,j})_{L^2(\Omega)},
  \end{equation}
  since $\eta_T$ is constant on each element $T\in\T_h$; for the other terms we use \eqref{eq:A3} and \eqref{eq:A4} (with $v=\eta_T$, $w=\nu\cdot\mscale\nabla\phi_{T,j}$ and $u=\phi_{T,j}$). 
  We can, thus, arrive to
  \begin{equation}\label{eq:decay7}
    \begin{aligned}
      & \Enormhw{\phi_{T,j}}{\Omega\setminus\omega^{M}_T}^2 \leq \Bh{\eta_T\phi_{T,j}}{\eta_T\phi_{T,j}} = \Bh{\phi_{T,j}}{\eta_T^2\phi_{T,j}} \\
      &\quad+ \sum_{e\in\E(\Omega)}\big(1/2(\{\nu\cdot\mscale\nabla\phi_{T,j}\},[\eta_T]^2[\phi_{T,j}])_{L^2(e)}-1/4 ([\nu\cdot\mscale\nabla\phi_{T,j}],[\eta_T]^2\{\phi_{T,j}\})_{L^2(e)} \\
      &\quad -\frac{\sigma}{4h_e}([\eta_T]^2,[\phi_{T,j}]^2)_{L^2(e)} +\frac{\sigma}{h_e}([\eta_T]^2,\{\phi_{T,j}\}^2)_{L^2(e)} \big), 
    \end{aligned}
  \end{equation}
  using \eqref{eq:decay5}, \eqref{eq:decaycutofbilinear}, and \eqref{eq:decay6}.
  Note that,
  \begin{equation}\label{eq:decay8}
    \begin{aligned}
      &\sum_{e\in\E(\Omega)}\big(1/2(\{\nu\cdot\mscale\nabla\phi_{T,j}\},[\eta_T]^2[\phi_{T,j}])_{L^2(e)}-1/4 ([\nu\cdot\mscale\nabla\phi_{T,j}],[\eta_T]^2\{\phi_{T,j}\})_{L^2(e)} \\
      &\qquad -\frac{\sigma}{4h_e}([\eta_T]^2,[\phi_{T,j}]^2)_{L^2(e)} + \frac{\sigma}{h_e}([\eta_T]^2,\{\phi_{T,j}\}^2)_{L^2(e)} \big) \\
      & \lesssim \hspace{-10pt}\sum_{\substack{e\in\E(\Omega):\\e\cap\omega^{M}_T\setminus\omega^{m}_T\neq 0}}\hspace{-10pt}\frac{h_T^2}{\ell^2 H_T^2}\Big(\|\{\nu\cdot\mscale\nabla\phi_{T,j}\}\|_{L^2(e)}\|[\phi_{T,j}]\|_{L^2(e)}+\|[\nu\cdot\mscale\nabla\phi_{T,j}]\|_{L^2(e)}\|\{\phi_{T,j}\}\|_{L^2(e)} \\
       &\qquad +\sigma\left(\|[\phi_{T,j}]\|^2_{L^2(e)}+\|\{\phi_{T,j}\}\|^2_{L^2(e)}\right) \Big) \\
      & \lesssim \hspace{-10pt}\sum_{\substack{e\in\E(\Omega):\\e\cap\omega^{M}_T\setminus\omega^{m}_T\neq 0 }}\hspace{-10pt}\big(\frac{h_T}{\ell^2 H_T^2}\|\mscale\nabla\phi_{T,j}\|_{L^2(T^+\cup T^-)}\|\phi_{T,j}\|_{L^2(T^+\cup T^-)} +\frac{\sigma}{\ell^2 H_T^2}\|\phi_{T,j}\|_{L^2(T^+\cup T^-)}^2 \big) \\
      & \lesssim \amax \ell^{-2} \|H^{-1}_T(\phi_{T,j}-\Lp\phi_{T,j})\|^2_{L^2(\omega^{M}_T\setminus\omega^{m}_T)} \leq \ba^2\ell^{-2}\Enormhw{\phi_{T,j}}{\omega^{M}_T\setminus\omega^{m}_T}^2.
     \end{aligned}
   \end{equation}
  Now we bound the term
   \begin{equation}\label{eq:decay9}
     \begin{aligned}
       \Bh{\phi_{T,j}}{\eta_T^2\phi_{T,j}} & = \Bh{\phi_{T,j}}{\eta_T^2\phi_{T,j}-b}+\Bh{\phi_{T,j}}{b}= \Bh{\phi_{T,j}}{b} \\
       &\lesssim \Enormhw{\phi_{T,j}}{\omega^{M}_T\setminus\omega^{m}_T}\Enormhw{b}{\omega^{M}_T\setminus\omega^{m}_T}\\
      & \leq \ba\Enormhw{\phi_{T,j}}{\omega^{M}_T\setminus\omega^{m}_T}\EnormHw{\Lp \eta^2_T\phi_{T,j}}{\omega^{M}_T\setminus\omega^{m}_T}. \\
    \end{aligned}
  \end{equation}
Furthermore, we have that
\begin{equation}
  \begin{aligned}\label{eq:decay10}
    &\EnormHw{\Lp \eta^2_T\phi_{T,j}}{\omega^{M}_T\setminus\omega^{m}_T} = \EnormHw{\Lp (\eta^2_T-\Pi_0(\T_H)\eta^2_T)\phi_{T,j}}{\omega^{M}_T\setminus\omega^{m}_T}\\
  &\quad =  \|\sqrt{\mscale}\nabla_H \Lp(\eta^2_T-\Pi_0(\T_H)\eta^2_T)\phi_{T,j}\|^2_{L^2(\omega^{M}_T\setminus\omega^{m}_T)} \\ 
  & \qquad+ \sum_{\substack{e\in\E(\Omega)\cup\E(\Gamma_D):\\e\cap\overline{\omega^{M}_T\setminus\omega^{m}_T}\neq 0 }} \frac{\sigma}{H_e}\|[\Lp(\eta^2_T-\Pi_0(\T_H)\eta^2_T)\phi_{T,j}]\|^2_{L^2(e)} \\
  &\quad\lesssim  \amax\|H_e^{-1}\Lp(\eta^2_T-\Pi_0(\T_H)\eta^2_T)\phi_{T,j}\|^2_{L^2(\omega^{M}_T\setminus\omega^{m}_T)} \\
  &\quad\lesssim  \amax\|H_e^{-1}(\eta^2_T-\Pi_0(\T_H)\eta^2_T)\|_{L^\infty(T)}\|\phi_{T,j}\|^2_{L^2(\omega^{M}_T\setminus\omega^{m}_T)} \\
  &\quad\lesssim \amax\ell^{-2}\|H_e^{-1}(\phi_{T,j}-\Lp\phi_{T,j})\|^2_{L^2(\omega^{M}_T\setminus\omega^{m}_T)}  \\
  &\quad\lesssim \ba^2\ell^{-2}\Enormhw{\phi_{T,j}}{\omega^{M}_T\setminus\omega^{m}_T}^2 ,
  \end{aligned}
\end{equation}
using a trace inequality, inverse inequality, and \lemref{lem:Lp_approx}, respectively. Combining the inequalities \eqref{eq:decay7}, \eqref{eq:decay8}, \eqref{eq:decay9}, and \eqref{eq:decay10} yields
\begin{equation}\notag
  \Enormhw{\phi_{T,j}}{\Omega\setminus\omega^{M}_T}^2 \leq \frac{ C_2}{\ell}\Enormhw{\phi_{T,j}}{\omega^{M}_T\setminus\omega^{m}_T}^2 \leq \frac{ C_2}{\ell} \Enormhw{\phi_{T,j}}{\Omega\setminus\omega^{m}_T}^2.
  \end{equation}
where $ C_2 = C'\ba^2$. Substituting back to $\ell$ and $k$ and using a cut off function with a slightly different argument, yields
  \begin{equation}\notag
    \begin{aligned}
    &\Enormhw{\phi_{T,j}}{\Omega\setminus\omega^{\ell k -1}_T}^2 \leq  \frac{ C_2}{\ell}\Enormhw{\phi_{T,j}}{\Omega\setminus\omega^{\ell (k-1)-1}_T}^2 \leq (\frac{ C_2}{\ell})^2\Enormhw{\phi_{T,j}}{\Omega\setminus\omega^{\ell (k-2)-1}_T}^2 \\
    &\quad\leq \dots \leq (\frac{ C_2}{\ell})^{k-1}\Enormhw{\phi_{T,j}}{\omega^{\ell}\setminus\omega^{\ell-1}_T}^2.
    \end{aligned}
  \end{equation}
which concludes the proof together with \eqref{eq:decay_11}.
\end{proof}
\begin{lemma}\label{lem:discrete_finescale_decay}
  For all, $T\in\T_H,\;j=1,\ldots,r$, the estimate
  \begin{equation}\notag
    \Enormh{\hspace{-2ex}\sum_{T\in\T_H,\;j=1,\ldots,r}\hspace{-2ex}v_j(\phi_{T,j}-\phi^L_{T,j})}^2 \leq  C_4L^d\hspace{-2ex}\sum_{T\in\T_H,\;j=1,\ldots,r}\hspace{-2ex}|v_j|^2\Enormh{\phi_{T,j}-\phi^L_{T,j}}^2.
  \end{equation}
  is satisfied, with $C_4 = C\ba^3$ and $C$ positive constant independent of the mesh ($h$ or $H$), of the patch size $L$, and of the diffusion matrix $\mscale$.
\end{lemma}
\begin{proof}
 Let $w = \sum_{T\in\T_H,\;j=1,\ldots,r}v_j(\phi_{T,j}-\phi^L_{T,j})$, and note that
  \begin{equation}\label{eq:discrete_finescale_decay1}
    \begin{aligned}
      \Bh{\phi_{T,j}-\lambda_{T,j}}{w-\zeta_Tw+b} &= 0, \\
      \Bh{\phi^L_{T,j}-\lambda_{T,j}}{w-\zeta_Tw+b} &= 0,
    \end{aligned}
  \end{equation}
  where $\zeta_{T}:=\zeta_T^{L+1,L+2}$, using \lemref{lem:property_Lp} and the property of the cut-off function.
  From \eqref{eq:discrete_finescale_decay1} follows that $\Bh{w}{(1-\zeta_T)w+b_{T}}=0$ for all $w\in\Vfh$.
  We obtain
  \begin{equation}\label{eq:decayfunc1}
    \begin{aligned}
      &\Enormh{\hspace{-2ex}\sum_{T\in\T_H,\;j=1,\ldots,r}\hspace{-2ex}v_j(\phi_{T,j}-\phi^L_{T,j})} \lesssim \hspace{-2ex}\sum_{T\in\T_H,\;j=1,\ldots,r}\hspace{-2ex}v_j\Bh{\phi_{T,j}-\phi^L_{T,j}}{w} \\
      &\quad = \hspace{-2ex}\sum_{T\in\T_H,\;j=1,\ldots,r}\hspace{-2ex}v_j\Bh{\phi_{T,j}-\phi^L_{T,j}}{\zeta_{T} w-b} \\
      &\quad \lesssim \hspace{-2ex}\sum_{T\in\T_H,\;j=1,\ldots,r}\hspace{-2ex}|v_j|\cdot\Enormh{\phi_{T,j}-\phi^L_{T,j}}\left(\Enormh{\zeta_{T} w}+\Enormh{b}\right)\\
      &\quad \lesssim \hspace{-2ex}\sum_{T\in\T_H,\;j=1,\ldots,r}\hspace{-2ex}|v_j|\cdot\Enormh{\phi_{T,j}-\phi^L_{T,j}}\left(\Enormh{\zeta_{T} w}+\ba\EnormH{\Lp \zeta_{T}w}\right)\\
      &\quad \lesssim \hspace{-2ex}\sum_{T\in\T_H,\;j=1,\ldots,r}\hspace{-2ex}|v_j|\cdot\Enormh{\phi_{T,j}-\phi^L_{T,j}}\ba^2\Enormh{\zeta_{T}w}\\
    \end{aligned}
  \end{equation}
  From \eqref{eq:decay3} and \eqref{eq:decay4}, we have
  \begin{equation}\label{eq:decayfunc2}
    \begin{aligned}
      \Enormh{\zeta_T w}=\Enormhw{\zeta_T w}{\omega^{L+2}_T}\lesssim \ba\Enormhw{ w}{\omega^{L+2}_T}.
    \end{aligned}
  \end{equation}
  Then, further estimation of \eqref{eq:decayfunc1} can be achieved using \eqref{eq:decayfunc2} and the discrete Cauchy-Schwarz inequality:
  \begin{equation}\notag
    \begin{aligned}
      &\Enormh{\hspace{-2ex}\sum_{T\in\T_H,\;j=1,\ldots,r}\hspace{-2ex}v_j(\phi_{T,j}-\phi^L_{T,j})}  \\
      & \quad\leq  \ba^3\left(\hspace{-0ex}\sum_{T\in\T_H,\;j=1,\ldots,r}\hspace{-4ex}|v_j|^2\Enormh{\phi_{T,j}-\phi^L_{T,j}}^2\right)^{\hspace{-1ex}1/2}\hspace{-1.7ex}\left(\hspace{-0ex}\sum_{T\in\T_H,\;j=1,\ldots,r}\hspace{-4ex}\Enormhw{w}{\omega^{L+2}_T}^2\right)^{\hspace{-1ex}1/2} \\
      & \quad\leq  \ba^3L^{d/2} \cdot\left(\hspace{-0ex}\sum_{T\in\T_H,\;j=1,\ldots,r}\hspace{-4ex}|v_j|^2\Enormh{\phi_{T,j}-\phi^L_{T,j}}^2\right)^{1/2}\cdot\Enormh{ w}.
    \end{aligned}
  \end{equation}
Dividing by $w$ on both sides concludes the proof.
\end{proof}

The following theorem gives an error bound for the dG multiscale method.
\begin{theorem}\label{thm:discrete}
Let $u\in H^1_D(\Omega)$ solve \eqref{e:modelvar} and $\umshHL\in\VmshHL$ solve \eqref{eq:DGMM}. Then, the estimate
  \begin{equation}\notag
    \begin{aligned}
    \Enormh{u-\umshHL} \leq &\Enormh{u-u_h} + C_1\amin^{-1/2}|| H(f-\Lp f ) ||_{L^2(\Omega)} \\ &+C_5\|H^{-1}\|_{L^\infty(\Omega)}L^{d/2}\gamma^L\|f\|_{L^2(\Omega)},
    \end{aligned}
  \end{equation}
  is satisfied, with $0<\gamma<1$, $L$ from \lemref{lem:decay_layers}, $C_1$ from \thmref{thm:umshH}, $C_5=C\ba^2C_4^{1/2}C_3$ where $C_3$ is from \lemref{lem:decay_layers} and $C_4$ is from \lemref{lem:discrete_finescale_decay}. $C$ a positive constant independent of the mesh ($h$ or $H$), of the patch size $L$, and of the diffusion matrix $\mscale$. 
\end{theorem}
\begin{rem}
  To counteract the factor $\|H^{-1}\|_{L^\infty(\Omega)}$ in the error bound in \thmref{thm:discrete}, we can choose the localization parameter as $L = \lceil C\log(||H^{-1}||_{L^\infty(\Omega)})\rceil$. On adaptively refined meshes it is recommended to choose $L=\lceil C\log(H^{-1})\rceil$.
\end{rem}
\begin{proof} We define $\tilde{u}^{ms,L}_{H}:=\sum_{T\in\T_H,\;j=1,\ldots,r}\umshHT(x_j)\phi^L_{T,j}$. Then, we obtain
  \begin{equation}\label{eq:discrete1}
    \begin{aligned}
      &\Enormh{u-\umshHL} \leq \Enormh{u-\tilde{u}^{ms,L}_{H}} \\
      & \leq \Enormh{u-u_h}+\Enormh{u_h-\umshH}+\Enormh{\umshH-\tilde{u}^{ms,L}_{H}} \\
       & \leq \Enormh{u-u_h}+\Enormh{u_h-\umshH}+\Enormh{\hspace{-2ex}\sum_{T\in\T_H,\;j=1,\ldots,r}\hspace{-2ex}\umshHT(x_j)(\phi_{T,j}-\phi^L_{T,j})}.
    \end{aligned}
  \end{equation}   
  Now, estimating the terms in \eqref{eq:discrete1}, we have 
  \begin{equation}\notag
    \Enormh{u_h-\umshH} \leq C_1\amin^{-1/2}\|H(f-\Lp f)\|_{L^2(\Omega)},
  \end{equation}
  using \thmref{thm:umshH}, and
  \begin{equation}
    \begin{aligned}\label{eq:discrete2}
      &\Enormh{\hspace{-2ex}\sum_{T\in\T_H,\;j=1,\ldots,r}\hspace{-2ex} \umshHT(x_j)(\phi_{T,j} - \phi^L_{T,j})}^2 \\ 
      &\leq  C_4L^d\hspace{-2ex}\sum_{T\in\T_H,\;j=1,\ldots,r}\hspace{-2ex}|\umshHT(x_j)|^2\Enormh{\phi_{T,j}-\phi^L_{T,j}}^2. \\
      & \leq C_4C_3^2L^d\gamma^{2L}\hspace{-2ex}\sum_{T\in\T_H,\;j=1,\ldots,r}\hspace{-2ex}|\umshHT(x_j)|^2\Enormh{\phi_{T,j}-\lambda_{j}}^2,
    \end{aligned}
  \end{equation}
  using \lemref{lem:discrete_finescale_decay} and \lemref{lem:decay_layers}, respectively.
  Further estimation, using \lemref{l:stabbasis}, yields
  \begin{equation}
    \begin{aligned}\label{eq:discrete3}
      &\hspace{-2ex}\sum_{T\in\T_H,\;j=1,\ldots,r}\hspace{-2ex}|\umshHT(x_j)|^2\Enormh{\phi_{T,j}-\lambda_{T,j}}^2\\
      &\quad\lesssim\ba^2\hspace{-2ex}\sum_{T\in\T_H,\;j=1,\ldots,r}\hspace{-2ex}|\umshHT(x_j)|^2\EnormH{\lambda_{T,j}}^2 \\
      &\quad\lesssim\ba^2\amax\hspace{-2ex}\sum_{T\in\T_H,\;j=1,\ldots,r}\hspace{-2ex}|\umshHT(x_j)|^2H_T^{-2}\|\lambda_{T,j}\|_{L^2(T)}^2 \\
      &\quad= \ba^2\amax\hspace{-2ex}\sum_{T\in\T_H,\;j=1,\ldots,r}\hspace{-2ex}\|H^{-1}_T\umshHT(x_j)\lambda_{T,j}\|_{L^2(T)}^2 \\
      &\quad\lesssim \ba^2\amax\|\hspace{-2ex}\sum_{T\in\T_H,\;j=1,\ldots,r}\hspace{-2ex}H^{-1}_T\umshHT(x_j)\lambda_{T,j}\|_{L^2(\Omega)}^2. \\
    \end{aligned}
  \end{equation}
  Furthermore, using a Poincare-Friedrich inequality for piecewise $H^1$ functions, we deduce
  \begin{equation}\label{eq:discrete4}
    \begin{aligned}
      &\|\hspace{-2ex}\sum_{T\in\T_H,\;j=1,\ldots,r}\hspace{-2ex}H_T^{-1}\umshHT(x_j)\lambda_{T,j}\|_{L^2(\Omega)}^2 \\
      &\lesssim\|\hspace{-2ex}\sum_{T\in\T_H,\;j=1,\ldots,r}\hspace{-2ex}H_T^{-1}\umshHT(x_j)\Lp(\lambda_{T,j}-\phi_{T,j})\|_{L^2(\Omega)}^2 \\
      & \lesssim \amin^{-1}\Enormh{H^{-1}\umshH}^2 \\
      &\lesssim \alpha^{-1}\|H^{-1}\|^2_{L^\infty(\Omega)}\|f\|_{L^2(\Omega)}^2.
    \end{aligned}
  \end{equation}
  Combining \eqref{eq:discrete2}, \eqref{eq:discrete3} and \eqref{eq:discrete4}, we arrive to
  \begin{equation}\notag
    \Enormh{\umshH-\umshHL} \lesssim \ba^2C_4^{1/2}C_3\|H^{-1}\|_{L^\infty(\Omega)}L^{d/2}\gamma^L\|f\|_{L^2(\Omega)}.
  \end{equation}
\end{proof}
\subsection{Error in a quantity of interest}\label{ss:quantity}
In engineering applications, we are often interested in a quantity of interest, usually a functional $g\in L^2(\Omega)$ of the solution. To this end, consider the dual reference solution \eqref{e:DGweak}: find $\phi_h\in\Vh$ such that
\begin{equation}\label{eq:dual_DGweak}
  a_h(v,\phi_h) = g(v)\quad\text{for all }v\in\V_h,
\end{equation}
and the dual multiscale solution \eqref{eq:DGMM}: find $\phi_{H,h}^L\in\V_{H,h}^L$ such that
\begin{equation}\label{eq:dual_DGMM}
  a_h(v,\phi_{H,h}^{ms,L}) = g(v)\quad\text{for all }v\in\V_{H,h}^L.
\end{equation}
\begin{theorem}\label{thm:qoi}
  Let $u\in H^1_D(\Omega)$ solve \eqref{e:modelvar}, $u_H^{ms,L}\in \V_H^{ms,L}$ solve \eqref{eq:DGMM}, and let $g\in L^2(\Omega)$ be the quantity of interest. Then, the estimate
  \begin{equation}\notag
    |g(u)-g(u_{H,h}^L)|\lesssim |g(u)-g(u_h)| + \Enormh{u_h - u_H^{ms,L}}\Enormh{\phi_h - \phi_H^{ms,L}},
  \end{equation}
  is satisfied.
\end{theorem}
\begin{proof}
From \eqref{eq:dual_DGweak} and \eqref{eq:dual_DGMM}, we obtain the Galerkin orthogonality
\begin{equation}\label{eq:dual_Galerkin_orthogonality}
  a_h(v,\phi_h-\phi_{H}^{ms,L}) = 0\quad\text{for all }v\in\V_H^{ms,L}.
\end{equation}
Using the triangle inequality, we have
\begin{equation}\notag
  |g(u)-g(u_H^{ms,L})|\leq|g(u)-g(u_h)|+|g(u_h)-g(u_H^{ms,L})|.
\end{equation}
Finally, observing that
\begin{equation}\notag
  \begin{aligned}
    |g(u_h - u_{h,H}^L)|&= |a_h(u_h - u_{h,H}^L,\phi_h)| \\
    & = |a_h(u_h - u_{h,H}^L,\phi_h - \phi_{H,h}^L)| \\
    & \lesssim \Enormh{u_h - u_{h,H}^L}\Enormh{\phi_h - \phi_{H,h}^L}, \\
  \end{aligned}
\end{equation}
using \eqref{eq:dual_Galerkin_orthogonality}, concludes the proof.
\end{proof}
\begin{corollary}\label{cor:L2bound}
 For $g(v)=(u_h-u_{H,h}^L,v)_{L^2(\Omega)}$, the following $L^2$-norm error estimates hold:
 \begin{equation}\notag
  \|u-u_{H,h}^L\|_{L^2(\Omega)} \lesssim \|u-u_h\|_{L^2(\Omega)}+\Enormh{u_h - u_H^{ms,L}}^{1/2}\Enormh{\phi_h - \phi_H^{ms,L}}^{1/2},
\end{equation}
and
 \begin{equation}\label{eq:L2est_2}
  \|u-u_{H,h}^L\|_{L^2(\Omega)} \lesssim \|u-u_h\|_{L^2(\Omega)}+H\Enormh{u_h - u_H^{ms,L}},
\end{equation}
for $L=\lceil C\log(H^{-1})\rceil$ with $C$ sufficiently large positive constant independent of the mesh parameters.
\end{corollary}
\begin{rem}
As expected, if we are interested in a smoother functional, a higher convergence rate is obtained for $|g(u_h-u_H^{ms,L})|$. For example, given the forcing function for the primal problem $f\in H^m(\T_H)$, a quantity of interest $g\in H^n(\T_H)$ (with $H^0(\T_H)$ denoting the standard $L^2(\Omega)$ space), and choosing $L=\lceil C\log(H^{-1}) \rceil$ with large enough $C$, gives
\begin{equation}\notag
  \begin{aligned}
    |g(u - u_{H}^{ms,L})| \lesssim |g(u)-g(u_h)| +  H^{2+m+n}(\sum_{T\in\T_H}|f|_{H^m(T)})(\sum_{T\in\T_H}|g|_{H^n(T)}).
  \end{aligned}
\end{equation}
\end{rem}
\section{Numerical Experiments}\label{s:numexp}
Let $\Omega$ where be an $L$-shaped domain (constructed by removing the lower right quadrant in the unit square) and let the forcing function be $f=1+\cos(2\pi x)\cos(2\pi y)$ for $(x,y)\in\Omega$. The boundary $\Gamma$ is divided into the Neumann boundary $\Gamma_N:=\Gamma\cap(\{(x,y): y=0\}\cup \{(x,y):  x=1\})$ and the Dirichlet boundary $\Gamma_D=\Gamma\setminus\Gamma_N$. We shall consider three different permeabilities: constant $\mscale_1=1$, $\mscale_2=\mscale_2(x)$, which is piecewise constant with periodic values of $1$ and $0.01$ with respect to a Cartesian grid of width $2^{-6}$ in the $x$-direction, and $\mscale_3 = \mscale_3(x,y)$ which piecewise constant with respect to a Cartesian grid of width $2^{-6}$ both in the $x$- and $y$-directions and has a maximum ratio $\amax/\amin = 4\cdot 10^{6}$. The data for $\mscale_3$ are taken from layer $64$ in the $SPE$ benchmark problem, see \texttt{http://www.spe.org/web/csp/}. The permeabilities $A_2$ and $A_3$ are illustrated in \fgref{fg:LshapeP}.
\begin{figure}[th!]
  \centering
  \subfigure[$\beta/\alpha = 10^2$]{
    \includegraphics[width=0.45\textwidth]{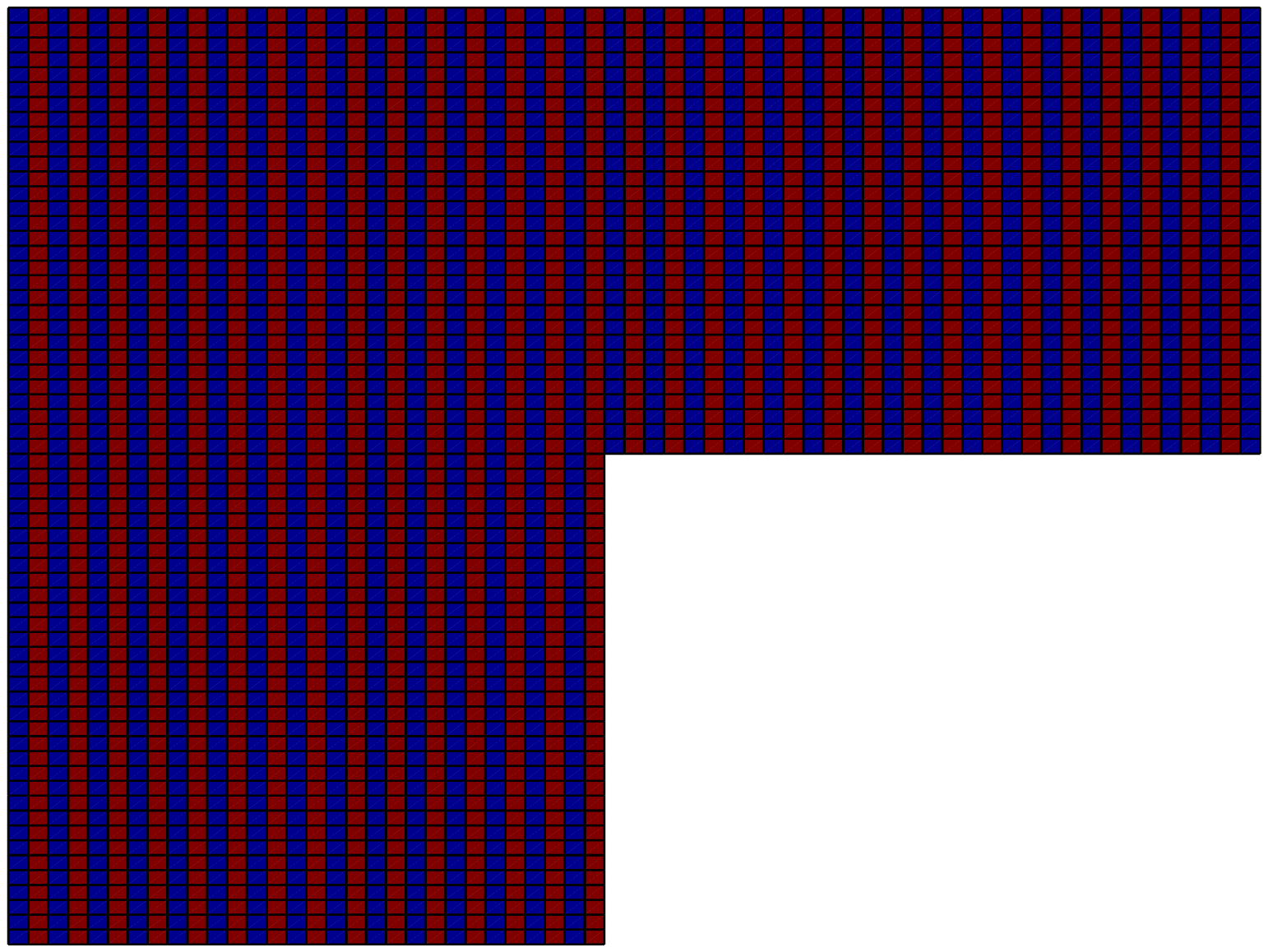}
  }
  \subfigure[$\beta/\alpha \approx 4\cdot 10^6$]{
    \includegraphics[width=0.45\textwidth]{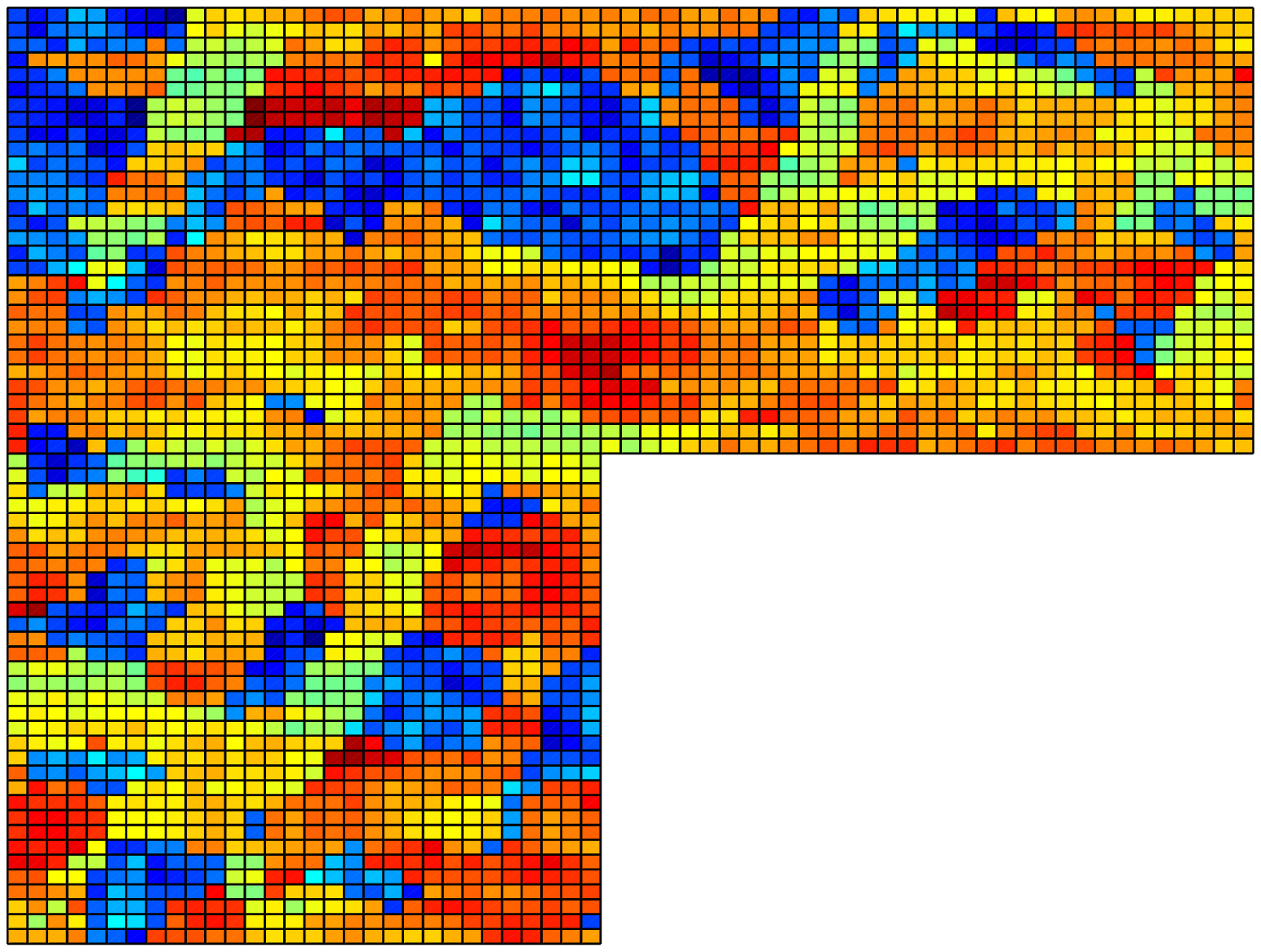}
  }
  \caption{The permeability structure of $A_2$ (a) and $A_3$ (b) in log scale.}
  \label{fg:LshapeP}
\end{figure}
For the periodic problem many of the corrected basis functions will be identical. For instance, all the local corrected basis in the interior are solved on identical patches, thereby reducing the computational effort considerably. In the extreme case of a problem with periodic coefficients on a unit hypercube, with period boundary conditions, the correctors $\phi_{T,j}$, $j=1,\dots,r$, will be identical for all $T\in\T_H$.

Consider the uniform (coarse) quadrilateral mesh $\T_H$ with size $H=2^{-i}$, $i=1,\dots,6$. The convergence rate $-p/2$ corresponds to $\mathcal{O}(H^p)$ since the number degrees of freedom $\approx H^{-2}$.
The corrector functions \eqref{eq:corrector} are solved on a subgrid of a (fine) quadrilateral mesh $\T_h$ with mesh size $2^{-8}$. The mesh $\T_h$ will also act as a reference grid for which we shall compute a reference solution $u_h\in\V_h$ \eqref{e:DGweak} on. Note that the mesh $\T_h$ is chosen so that it resolves the fine scale features for $\mscale_i$, $i=1,2,3$, hence we assume that the solution $u_h$ is sufficiently accurate. 

\subsection{Localization parameter}
If $f \in H^{m}(\T_H)$ we have the bound
\begin{equation}\label{eq:reg_rhs}
  ||H(f-\Lp f))||_{L^2(\Omega)} \lesssim \sum_{T\in\T_H} H^{k+1}|f|_{H^k(T)},
\end{equation}
where $k=0$ for $m=0$, $k=1$ for $m=1$, and $k=2$ for $m>1$. Hence, to balance the error in between the terms on the right-hand side of the estimate in \thmref{thm:discrete}, different constant $C$ has to be used for the localization parameter, $L=\lceil C\log(H^{-1})\rceil$, depending on the element-wise regularity of the forcing function $f$ on $\T_H$. 
\begin{figure}[th!]
  \centering
  \includegraphics[width=0.75\textwidth]{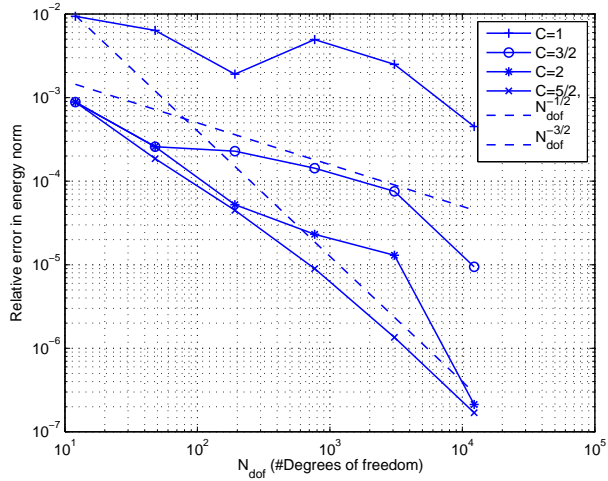}
  \caption{Diffusion coefficient $A_1 = 1$. Relative energy-norm error against $N_{\text{dof}}$, for different values of $C$ for the localisation parameter $L$.}
\label{fg:localization_parameter}
\end{figure}
\begin{figure}[th!]
  \centering
  \includegraphics[width=0.75\textwidth]{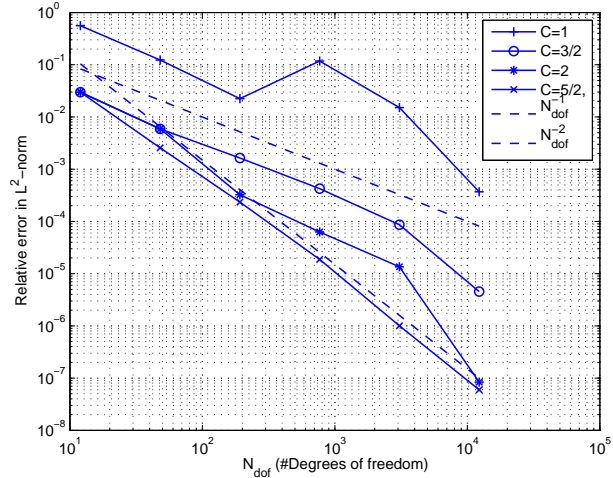}
  \caption{Diffusion coefficient $A_1 = 1$. Relative $L^2$-norm error against $N_{\text{dof}}$, for different values of $C$ for the localisation parameter $L$.}
\label{fg:localization_parameter_L2}
\end{figure}
\fgref{fg:localization_parameter} shows the relative error in the energy norm $\Enormh{u_h-u_{H}^{ms,L}}/\Enormh{u_h}$ and \fgref{fg:localization_parameter_L2} the relative error in the $L^2$-norm $\|u_h - u_{H}^{ms,L}\|_{L^2(\Omega)}/\|u_h\|_{L^2(\Omega)}$ between $u_h$ and $u_H^{ms,L}$ against the number of degrees of freedom $N_{\text{dof}}\approx O(H^{-2})$, using different constants $C=1,3/2,2,5/2$. With the chose $C=5/2$, the errors due to the localization can be neglected compared to the errors from the forcing function, both for the energy- and for the $L^2$-norm. For $f\not\in H^1(\T_h)$, $C=3/2$ is sufficient since \eqref{eq:reg_rhs} gives linear convergence. In the following numerical experiments we shall use $C=2$, since this value seems to balance the error sufficiently. Note that the numerical overhead increases with $C$ as the sizes of the patches $\omega^L_T$ $T\in\T_H$, increases with $L=\lceil C\log(H^{-1})\rceil$. This results in both increased computational effort to compute the corrector functions and reduced sparseness in the coarse scale stiffness matrix.

\subsection{Energy-norm convergence}
Let the localization parameter be \\$L=\lceil 2\log(H^{-1})\rceil$. \fgref{fg:conv} shows the relative error in the energy norm plotted against the number of degrees of freedom.
The different permeabilities $A_i$, $i=1,2,3$, and the singularity arising from the $L$-shaped domain do not appear to have a substantial impact on the convergence rate, which is about  $3/2$, as expected. We note in passing that using standard dG on the coarse mesh only admits poor convergence behaviour for $A_2$ and for $A_3$. This is to be expected, since standard dG on the coarse mesh does not resolve the fine scale features. 
\begin{figure}[th!]
  \centering
 \includegraphics[width=0.75\textwidth]{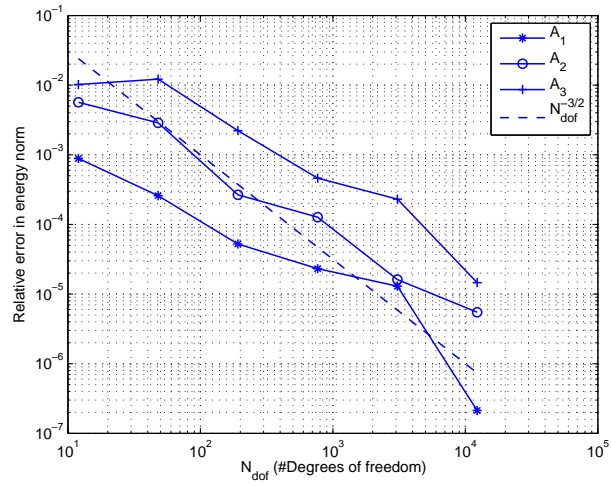}
 \caption{Relative energy-norm error against $N_{\text{dof}}$, for $C=2$ in the localisation parameter $L$ for the the diffusion coefficients $A_1$, $A_2$, and $A_3$.}
 \label{fg:conv}
\end{figure}

\subsection{$L^2$-norm convergence}
Again, set $L=\lceil 2\log(H^{-1})\rceil$. \fgref{fg:conv_L2} and \fgref{fg:conv_TL2}, shows the relative $L^2$-norm error again the number of degrees of freedom between, $u_h$ and $u_{H,h}^L$ and between $u_h$ and $\Lp u_{H,h}^L$, viz., $\|u_h-\Lp u_{H}^{ms,L}\|_{L^2(\Omega)}/\|u_h\|_{L^2(\Omega)}$, respectively. In \fgref{fg:conv_L2} we see that the $L^2$-norm error between $u_h$ and $u_H^{ms,L}$ converges at a faster rate than in the energy norm (convergence rate $-2$ compared to $-3/2$, respectively,) as expected from \eqref{eq:L2est_2}. In \fgref{fg:conv_TL2} only the coarse part of $u_H^{ms,L}$ is used (i.e. $\Lp u_H^{ms,L}$); nevertheless it appears to have a faster convergence rate than $-1/2$,  except for the case of the permeability $\mscale_3$.
\begin{figure}[th!]
  \centering
 \includegraphics[width=0.75\textwidth]{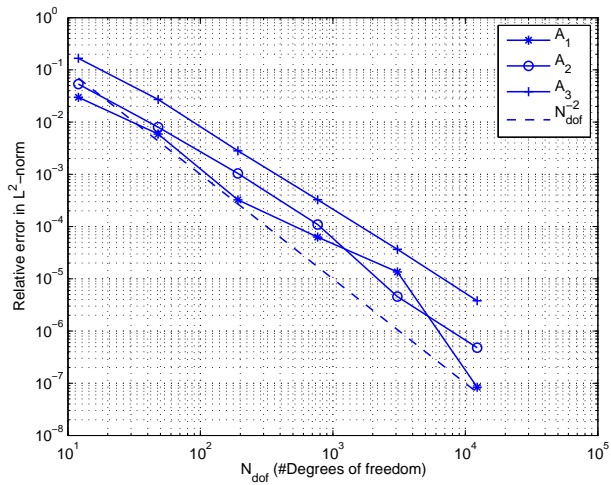}
 \caption{Relative $L^2$-norm error against $N_{\text{dof}}$, for $C=2$ in the localisation parameter $L$ for the the diffusion coefficients $A_1$, $A_2$, and $A_3$.}
 \label{fg:conv_L2}
\end{figure}
\begin{figure}[th!]
  \centering
 \includegraphics[width=0.75\textwidth]{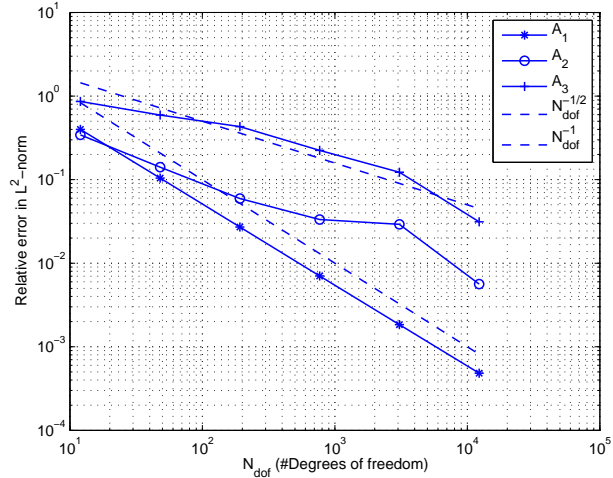}
 \caption{Relative $L^2$-norm error against $N_{\text{dof}}$, for $C=2$ in the localisation parameter $L$ for the the diffusion coefficients $A_1$, $A_2$, and $A_3$.}
\label{fg:conv_TL2}
\end{figure}

\bibliographystyle{amsplain}
\bibliography{references}

\appendix

\section{Equalities for averages and jump operators}
We derive equalities for averages and jump operators across interfaces where the functions $v$ and $w$ have discontinuities.
  Using $[vw] = \{v\}[w] + [v]\{w\}$ and $\{v\}\{w\} = \{vw\}-1/4[v][w]$, we have
  \begin{equation}
    \begin{aligned}\label{eq:A1}
       \{v w\}[vu] = & \{w\}\{v\}[vu]+1/4[v][w][vu]\\
       = &\{w\}[v^2u]-\{w\}[v]\{vu\}+1/4[v][w][vu] \\
       = &\{w\}[vu]-[v]\{w\}\{v\}\{u\}-1/4[v]\{w\}[v][u] \\
       &\quad +1/4[v]^2[w]\{u\}+1/4[v]\{v\}[w][u]
    \end{aligned}
  \end{equation}
  and
  \begin{equation}
    \begin{aligned}\label{eq:A2}
       \{vw\}[vu] =& \{v\}\{vw\}[u] + \{vw\}[v]\{u\} \\
       = &\{v^2w\}[u]-1/4[v][vw][u] + \{vw\}[v]\{u\} \\
       = &\{v^2w\}[u] -1/4[v]^2\{w\}[u]-1/4[v]\{v\}[w][u] \\
       &\quad+ [v]\{v\}\{w\}\{u\}+1/4 [v]^2[w]\{u\} \\
    \end{aligned}
  \end{equation}
  Combining \eqref{eq:A1} and \eqref{eq:A2} we obtain
  \begin{equation}
    \begin{aligned}\label{eq:A3}
      2\{vw\}[vu] = \{w\}[v^2u]+\{v^2w\}[u]+1/2 [v]^2[w]\{u\}-1/2[v]^2\{w\}[u]
    \end{aligned}
  \end{equation}
  Also,
  \begin{equation}
    \begin{aligned}\label{eq:A4}
      [vu][vu] =& [u]\{v\}[vu]+[v]\{u\}[vu] \\
      =& [u][v^2u]-[v][u]\{vu\}+[v]\{u\}[vu] \\
      =& [u][v^2u]) -[v][u]\{v\}\{u\}-1/4[v][u][v][u] \\
      & \quad +[v]\{u\}[v]\{u\}+[v]\{u\}\{v\}[u] \\
      =& [u][v^2u] -1/4[v]^2[u]^2 +[v]^2\{u\}^2
    \end{aligned}
  \end{equation}
\end{document}